\documentclass[a4paper]{amsart}
\usepackage{aliascnt, amssymb, enumerate, enumitem, hyperref, mathtools}
\usepackage[all]{xy}

\def\today{\number\day\space\ifcase\month\or   January\or February\or
   March\or April\or May\or June\or   July\or August\or September\or
   October\or November\or December\fi\   \number\year}

\theoremstyle{definition}

\newaliascnt{thmCt}{lma}
\newtheorem{thm}[thmCt]{Theorem}
\aliascntresetthe{thmCt}

\newtheorem*{thm*}{Theorem}

\newaliascnt{corCt}{lma}

\aliascntresetthe{corCt}

\newaliascnt{propCt}{lma}
\newtheorem{prop}[propCt]{Proposition}
\aliascntresetthe{propCt}

\newaliascnt{pgrCt}{lma}

\aliascntresetthe{pgrCt}

\newaliascnt{dfCt}{lma}
\newtheorem{df}[dfCt]{Definition}
\aliascntresetthe{dfCt}

\newaliascnt{remCt}{lma}
\newtheorem{rem}[remCt]{Remark}
\aliascntresetthe{remCt}

\newaliascnt{remsCt}{lma}

\aliascntresetthe{remsCt}

\newaliascnt{egCt}{lma}
\newtheorem{eg}[egCt]{Example}
\aliascntresetthe{egCt}

\newaliascnt{egsCt}{lma}

\aliascntresetthe{egsCt}

\newaliascnt{qstCt}{lma}

\aliascntresetthe{qstCt}

\newaliascnt{pbmCt}{lma}

\aliascntresetthe{pbmCt}

\newaliascnt{notaCt}{lma}

\aliascntresetthe{notaCt}

\newaliascnt{cnjCt}{lma}

\aliascntresetthe{cnjCt}

\newcounter{theoremintro}
\newtheorem{thmintro}[theoremintro]{Theorem}

\newcommand{\N}{{\mathbb{N}}}
\newcommand{\T}{{\mathbb{T}}}

\newcommand{\supp}{{\mathrm{supp}}}

\newcommand{\andeqn}{\,\,\, {\mbox{and}} \,\,\,}

\newcommand{\ct}{continuous}

\newcommand\interior[1]{{#1}^{\circ}}
\newcommand{\ov}{\overline}

\title{An embedding version of Rubin's theorem}

\date{\today}
\thanks{The author was supported by Kungl. Vetenskapsakademiens stiftelser.
}

\author{Jan Gundelach}
\address[Jan Gundelach]{Department of Mathematical Sciences, Chalmers University of
Technology and University of Gothenburg, Gothenburg SE-412 96, Sweden.}
\email[]{jangund@chalmers.se}
\urladdr{www.chalmers.se/en/persons/jangund/}

\begin{document}

\begin{abstract}
    Rubin's theorem asserts that if $\Gamma\curvearrowright X$ and $\Delta\curvearrowright Y$ are Rubin actions, then any group isomorphism $\Gamma \cong \Delta$ induces an equivariant homeomorphism $Y\cong X$. We provide an embedding version of Rubin's theorem highlighting group embeddings that induce a spatial equivariant map of a certain form. We further showcase instances of such embeddings between generalized Brin-Thompson groups.
\end{abstract}
\maketitle

\section{Introduction}
In the 1980s, Rubin studied rigidity phenomena of faithful group actions $\Gamma \curvearrowright X$ on various structured objects $X$. The theme is a reconstruction of these objects from their group of transformations. Among other applications, his results also apply to topological spaces and highlight a class of faithful actions $\Gamma \curvearrowright X$ that allow for a reconstruction of the topological space $X$ and that we will refer to as \textit{Rubin actions} in the following; see \autoref{Rubin df}. Concretely, with the terminology of \cite{BelEllMat_Rubin_2022} and \cite{Mat_tfg_2015}, Rubin's theorem \cite[Corollary 3.5]{Rub_reconstruction_1989} states that if $\Gamma$ is a group admitting Rubin actions $\Gamma\curvearrowright X$ and $\Gamma \curvearrowright Y$, then there is a unique $\Gamma$-equivariant homeomorphism $Y\cong X$.\par
Classes of groups that admit a canonical Rubin action include topological full groups of minimal effective \'etale groupoids with Cantor unit space, which received more and more attention recently, both in topological dynamics, operator algebras, and in the structure theory of ample groupoids in general.\par Another way of interpreting Rubin's theorem is that, given two Rubin actions $\Gamma \curvearrowright X$ and $\Delta \curvearrowright Y$, a group isomorphism $\Phi\colon \Gamma \to \Delta$ induces a unique spatial homeomorphism $\rho\colon Y\to X$ such that the following diagram commutes for every $\gamma \in \Gamma$:
\[\xymatrix{X \ar[d]_-{\gamma}& Y \ar[l]_-{\rho} \ar[d]^-{\Phi(\gamma)}\\ X & Y.\ar[l]_-{\rho}}\]
In this language, another question asked by Rubin on \cite[p. 493]{Rub_reconstruction_1989} is whether there are any reasonable assumptions on a group \textit{embedding} $\Phi\colon \Gamma \to \Delta$, instead of a group isomorphism, so that $\Phi$ still induces a $\Phi$-equivariant spatial map in the sense of the commuting diagram above. In this paper, we address this question by using the framework and notation from \cite{BelEllMat_Rubin_2022} and \cite{Mat_tfg_2015}. Our main result reads as follows.

\begin{thmintro}\label{Rubin emb intro} (See \autoref{rubin emb char}.)
     Let $X$ and $Y$ be locally compact Hausdorff spaces with no isolated points, let $\Gamma \curvearrowright X$ be a Rubin action, let $\Delta\curvearrowright Y$ be a faithful action, and let $\Phi\colon \Gamma \rightarrow \Delta$ be an injective group homomorphism. The following are equivalent:
    \begin{enumerate}
        \item $\Phi(\Gamma)\curvearrowright Y$ has no global fixed points, is locally dense on saturated subsets, and satisfies $\Phi(\Gamma_{\textup{rsupp}_X(\gamma)}) = \Phi(\Gamma)_{\textup{rsupp}_Y(\Phi(\gamma))}$ for all $\gamma\in \Gamma$.
        \item There is a unique \ct, surjective and $\Phi$-equivariant spatial map $\rho\colon Y\to X$ such that $\rho^{-1}(\textup{rsupp}_X(\gamma)) = \textup{rsupp}_Y(\Phi(\gamma))$ for all $\gamma \in \Gamma$.
    \end{enumerate}
    \end{thmintro}
We call an embedding $\Phi$ as in $(1)$ a \textit{Rubin embedding} and the map $\rho$ as in $(2)$ its \textit{anchor map}. 

\section{Spatial realization}\label{spatial realization chapter}
In this section, we follow the construction in \cite[Section 3]{Mat_tfg_2015} and in particular of \cite[Theorem 3.5]{Mat_tfg_2015} in the more general framework of \cite{BelEllMat_Rubin_2022} and of a group embedding that is not necessarily an isomorphism.

\begin{df}\label{localized subgroups}
    Let $\Gamma$ be a group acting on a topological space $X$. The \textit{open support} of $\gamma\in \Gamma$ is given by $\supp'_X(\gamma) := \{x\in X\colon \gamma \cdot x \neq x\}$, the \textit{support} by $\supp_X(\gamma) := \overline{\supp'_X(\gamma)}$ and the \textit{regular support} by $\textup{rsupp}_X(\gamma) := \interior{\supp_X(\gamma)}$. For an open set $U\subseteq X$, we define the associated \textit{localized subgroup} as
    \[ \Gamma_U := \{ \gamma \in \Gamma\colon \,\,\supp'_X(\gamma)\subseteq U\}.\]
\end{df}

\begin{df}\label{Rubin df}
    Let $\Gamma$ be a group and let $X$ be a Hausdorff topological space. A faithful action $\Gamma \curvearrowright X$ is called 
    \begin{enumerate}
        \item \textit{locally dense}, if for all open $U\subseteq X$ and all $p\in U$, we have $p\in \interior{(\overline{\Gamma_U \cdot p})}$;
        \item \textit{Rubin}, if it is locally dense and $X$ is locally compact with no isolated points;
        \item \textit{locally moving}, if $\Gamma_U \neq \{1\}$ for all nonempty open sets $U\subseteq X$.
    \end{enumerate}
\end{df} 

\begin{rem}\label{lc moving weaker}
    Note that any Rubin action is locally moving because if $\Gamma_U = \{1\}$ for some nonempty open set $U\subseteq X$, then local density would imply that any $p\in U$ would be isolated. However, unlike locally dense actions, locally moving actions may have nowhere dense orbits under localized subgroups.
\end{rem}

\begin{eg}\label{thompson Rubin examples}
    The canonical faithful action of Thompson's group $F\curvearrowright [0,1]$ is locally moving, but admits global fixed points at $0$ and $1$. The faithful actions of Thompson's groups $F\curvearrowright (0,1),\, T\curvearrowright \T$ and $V\curvearrowright \{0,1\}^{\N}$ are examples of Rubin actions $\Gamma\curvearrowright X$, as we show next. All the ambient spaces are locally compact and have no isolated points. In the first two cases, any $p\in U$ is contained in an open dyadic interval $p\in D\subseteq U$ and by scaling partitions of $X$ to $D$ and leaving $X\setminus \ov{D}$ fixed, it follows that $\Gamma_D\curvearrowright D$ and $\Gamma\curvearrowright X$ are conjugate. Now minimality shows $\ov{\Gamma_D\cdot p} = \ov{D}$ and hence the actions are locally dense.\par The third action on the Cantor space is an instance of a bigger family of Rubin actions treated in \autoref{sfts}.
\end{eg}

\begin{df}\label{Brin-Thompson}
    Let $m\in \N$ and let $k_1,\dots k_m\in \N_{\geq 2}$ be alphabet sizes. For each $i=1,\dots,m$, we set $X_{k_i}:=\{0,\dots,k_i-1\}^\N$ and denote its elements as tuples of infinite strings without separators. A \textit{table} for the Cantor set $X=\prod_{i=1}^m X_{k_i}$ of size $l\in \N$ is a matrix of the form \[\begin{pmatrix} v\\ u\end{pmatrix}= \begin{pmatrix} v^{(1)}& v^{(2)}& \cdots & v^{(l)} \\ u^{(1)}& u^{(2)}& \cdots & u^{(l)}\end{pmatrix}\]
    with entries which are tuples of finite words $v^{(j)},\,u^{(j)}\in \prod_{i=1}^m\{0,\dots,k_i-1\}^*$ for all $j=1,\dots,l$ such that both rows describe a partition of \[X = \bigsqcup_{j=1}^l\prod_{i=1}^m u_i^{(j)}X_{k_i} = \bigsqcup_{j=1}^l\prod_{i=1}^m v_i^{(j)}X_{k_i}.\] That is, for every $y\in X$, there is a unique index $j$ and tail $x\in X$ such that $y=v^{(j)}x$. Every table induces a homeomorphism of $X$ by replacing the $v$-initial word tuple with the one of $u$, that is,
    \[ \overline{\begin{pmatrix} v\\ u\end{pmatrix}}\left(v^{(j)}x\right) := u^{(j)}x.
    \]
    The composition of two such homeomorphisms is induced by a refined table of possibly greater size and we call the thereby formed group of induced homeomorphisms the \textit{generalized Brin-Thompson group} $V_{k_1,\dots,k_m}$:
    \[ V_{k_1,\dots,k_m} := \Big\{\overline{\begin{pmatrix} v\\ u\end{pmatrix}}\in \textup{Homeo}(X)\colon \begin{pmatrix} v\\ u\end{pmatrix} \textup{\, is a table}\Big\}.
    \]
\end{df}
The name is motivated by the definition of the \textit{Brin-Thompson group} $mV_k$ in the case $k_1=\dots=k_m =k$. They arise as topological full groups of products of shifts of finite type (SFTs). The argument in the following example applies to all canonical actions of topological full groups associated to minimal effective \'etale groupoids with Cantor unit space, but we limit ourselves to the generalized Brin-Thompson groups for clarity. See also \cite[Definition 6.8, Proposition 7.7]{GarGun_lpemb_2026}.

\begin{eg}\label{sfts}
    The canonical faithful actions $V_{k_1,\dots,k_m}\curvearrowright \prod_{i=1}^m X_{k_i}$ of generalized Brin-Thompson groups are Rubin actions. To verify local density at $p$, without loss of generality, we can assume that the open neighborhood $p\in U$ is a cylinder set $\mu\prod_{i=1}^m X_{k_i}$ for some $\mu\in \prod_{i=1}^m\{0,\dots,k_i-1\}^*$. Then the associated localized subgroup action $(V_{k_1,\dots,k_m})_U\curvearrowright U$ is conjugate to $V_{k_1,\dots,k_m}\curvearrowright \prod_{i=1}^m X_{k_i}$ via
    \[ \iota_U\colon V_{k_1,\dots,k_m} \to (V_{k_1,\dots,k_m})_U,\quad \iota_U(\psi)(y):= \begin{cases}
        \mu\psi(x), \,\, \textup{if}\,\,y= \mu x\in \mu\prod_{i=1}^m X_{k_i};\\ y,\,\, \textup{otherwise.}
    \end{cases}
    \] Since $V_{k_1,\dots,k_m}\curvearrowright \prod_{i=1}^m X_{k_i}$ is minimal, we have $U = \interior{\ov{(V_{k_1,\dots,k_m})_U.p}}$ showing local density.
\end{eg}
The first out of two key steps in the proof of Rubin's theorem is to use the locally moving property to describe localized subgroups of regular open subsets algebraically and independently from the ambient space. To make this precise, we introduce the notion of algebraic disjointness.

\begin{df}\label{alg dj df}
    Let $\Gamma$ be a group, let $ f,g\in \Gamma$ and let $C_\Gamma(g) := \{\gamma\in \Gamma\colon \gamma g = g\gamma\}$ denote the centralizer of $g$. We say that $g$ is \textit{algebraically disjoint} from $f$ if for all $h\in \Gamma\setminus C_\Gamma(f)$ there are $f_1,\,f_2\in C_\Gamma(g)$ such that $[f_1,[f_2,h]]\in C_\Gamma(g)\setminus \{1\}$. We write $g\lhd^\Gamma_{alg} f$.
\end{df}
The relation $\lhd^\Gamma_{alg}$ on $\Gamma$ is neither symmetric, reflexive, nor preserved under group homomorphisms in general. However, despite its bad permanence properties, for locally moving actions, the following proposition justifies viewing the algebraic disjointness relation as a way to describe the lattice of regular localized subgroups in $\Gamma$ independently from $X$.\par
Recall the notion of locally moving in \autoref{Rubin df}. The following is \cite[Proposition 2.1]{BelEllMat_Rubin_2022}.

\begin{prop}\label{alg supp}
    Let $\Gamma \curvearrowright X$ be locally moving and let $f\in \Gamma$. Then $\Gamma_{\textup{rsupp}_X(f)} = C_\Gamma(\{g^{12}\colon g \lhd^\Gamma_{alg} f\})$.
\end{prop}
Given faithful actions $\Gamma \curvearrowright X$ and $\Delta \curvearrowright Y$, we may treat the acting groups as $\Gamma\subseteq\textup{Homeo}(X)$ and $\Delta \subseteq \textup{Homeo}(Y)$. To relate these actions, note that a group homomorphism $\Phi\colon \Gamma \to \Delta$ encodes that $\Delta \curvearrowright Y$ restricts to an action $\Gamma\curvearrowright Y$ and this restriction is a faithful action if and only if $\Phi$ is injective.

\begin{df}\label{lr def}
    Let $\Gamma \curvearrowright X$ and $\Delta \curvearrowright Y$ be actions on topological spaces and let $\Phi\colon \Gamma \to \Delta$ be a group homomorphism. Then $\Phi$ is called \textit{locally regular} with respect to $\Gamma \curvearrowright X$ and $\Delta \curvearrowright Y$ if $\Phi(\Gamma_{\textup{rsupp}_X(\gamma)}) = \Phi(\Gamma)_{\textup{rsupp}_Y(\Phi(\gamma))}$ for all $\gamma\in \Gamma$.
\end{df}

\begin{rem}\label{algebraic picture of lr}
    If both actions are locally moving, then local regularity does not depend on the ambient actions on $X$ or $Y$ by \autoref{alg supp}, and is just an algebraic condition met by isomorphisms, but not necessarily by every injective group homomorphism. In fact, for given $\gamma\in \Gamma$, neither containment of the two subgroups $\Phi(\Gamma_{\textup{rsupp}_X(\gamma)})$ and $\Phi(\Gamma)_{\textup{rsupp}_Y(\Phi(\gamma))}$ is automatic. By \autoref{alg supp} and injectivity of $\Phi$, the former group is $C_{\Phi(\Gamma)}(\{\Phi(\tau)^{12}\colon \tau \lhd^\Gamma_{alg} \gamma\})$, the latter is $C_{\Phi(\Gamma)}(\{\delta^{12}\colon \delta \lhd^\Delta_{alg} \Phi(\gamma)\})$, and neither of the two properties $\tau\lhd_{alg}^\Gamma \gamma$ or $\Phi(\tau)\lhd_{alg}^\Delta \Phi(\gamma)$ is implied by the other. 
\end{rem}
We continue to show that embeddings that admit a spatial map of the desired form are necessarily locally regular. 

\begin{df}\label{anchor for phi df}
    Let $\Gamma\subseteq\textup{Homeo}(X)$ and $\Delta \subseteq \textup{Homeo}(Y)$ be subgroups and let $\Phi\colon \Gamma \to \Delta$ be a group homomorphism. We say that a map $\rho\colon Y\to X$ is \textit{$\Phi$-equivariant} if $\rho\circ \Phi(\gamma) = \gamma \circ \rho$ for all $\gamma\in \Gamma$. We moreover say that a continuous $\Phi$-equivariant $\rho\colon Y\to X$ is an \textit{anchor map} for $\Phi$ if $\rho^{-1}(\textup{rsupp}_X(\gamma)) = \textup{rsupp}_Y(\Phi(\gamma))$ for all $\gamma \in \Gamma$.
\end{df}

\begin{prop}\label{spatial map implies lr}
    Let $\rho$ be an anchor map for $\Phi$. If $\Phi$ is injective and if $\rho$ is surjective, then $\Phi$ is locally regular. Furthermore, if $\Gamma\curvearrowright X$ is Rubin (see \autoref{Rubin df}), then there is at most one anchor map for $\Phi$.
\end{prop}

\begin{proof}
    The property $\Phi(\tau)\in \Phi(\Gamma_{\textup{rsupp}_X(\gamma)})$ is equivalent to $\textup{rsupp}_X(\tau)\subseteq \textup{rsupp}_X(\gamma)$ by injectivity of $\Phi$. This is further equivalent to $\rho^{-1}(\textup{rsupp}_X(\tau))\subseteq \rho^{-1}(\textup{rsupp}_X(\gamma))$ by surjectivity of $\rho$ and thus to $\textup{rsupp}_Y(\Phi(\tau))\subseteq \textup{rsupp}_Y(\Phi(\gamma))$ by assumption. This in turn just means that $\Phi(\tau)\in \Phi(\Gamma)_{\textup{rsupp}_Y(\Phi(\gamma))}$. For the second claim, let $\Gamma\curvearrowright X$ be Rubin and assume that there are two anchor maps $\rho_1$ and $\rho_2$ for $\Phi$. In order to reach a contradiction, assume that there exists $y\in Y$ such that $\rho_1(y)\neq \rho_2(y)$. Since regular supports form a basis for the topology on $X$ by \cite[Proposition 3.2]{BelEllMat_Rubin_2022} and since $X$ is Hausdorff, there is a neighborhood $\textup{rsupp}_X(\gamma)$ for some $\gamma\in \Gamma$ containing $\rho_1(y)$, but not $\rho_2(y)$. Hence $y\in \rho_1^{-1}(\textup{rsupp}_X(\gamma)) = \textup{rsupp}_Y(\Phi(\gamma))$ and $y\not\in \rho_2^{-1}(\textup{rsupp}_X(\gamma)) = \textup{rsupp}_Y(\Phi(\gamma))$. This is a contradiction, which shows that $\rho_1=\rho_2$.\qedhere
\end{proof}
The second key step in the proof of Rubin's theorem is to recover the space $X$ from the isomorphism class of $\Gamma$ by expressing points in $X$ algebraically as limits of ultrafilters consisting of regular supports. This information is contained in the poset of localized subgroups, and since these subgroups are preserved under group isomorphisms, this finally allows to construct the desired spatial map. To pursue a similar strategy for sufficiently regular group embeddings, we introduce more notation.

\begin{df}\label{filter df}
    Let $(\mathcal{R},\leq)$ be a poset. A nonempty subset $\mathcal{P}\subseteq \mathcal{R}$ is called a \textit{prefilter} on $\mathcal{R}$ if for all $U,V\in \mathcal{P}$, there is $W\in \mathcal{P}$ with $W\leq U$ and $W\leq V$. By Zorn's lemma, any prefilter $\mathcal{P}$ is contained in a maximal prefilter $\mathcal{F}$, which is called an \textit{ultrafilter}. \par In the special case that $(\mathcal{B},\subseteq)$ is a poset of open sets under inclusion that form a neighborhood basis for some topological space $X$, we say that an ultrafilter $\mathcal{F}\subseteq \mathcal{B}$ \textit{converges} to a point $x\in X$ if every neighborhood of $x$ contains a set from $\mathcal{F}$.
\end{df}

\begin{df}\label{saturated top}
    Let $\Gamma \curvearrowright X$ be an action on a topological space. We say that $\Gamma \curvearrowright X$ is \textit{of full support} if it has no global fixed points. We call the coarser topo\-logy on $X$ generated by 
    \[\Bigg\{ \bigcap_{j=1}^n \textup{supp}'_X(\gamma_j)\subseteq X\colon \gamma_1,\dots,\gamma_n\in \Gamma \andeqn \bigcap_{j=1}^n \textup{supp}'_X(\gamma_j) \neq \emptyset \Bigg\}\] the \textit{saturated topology}. We further refer to open subsets with respect to the saturated topology as \textit{saturated} open subsets and denote closures and interiors with respect to this topology by an indexed \textit{Sat}. For $\gamma\in \Gamma$, the \textit{saturated regular support} is given by $\textup{srs}_X(\gamma) := \interior{\ov{\textup{supp}'_X(\gamma)}^{Sat}}$. Finally, we define \[\mathcal{R}_X := \Bigg\{ \bigcap_{j=1}^n \textup{srs}_X(\gamma_j)\subseteq X\colon \gamma_1,\dots,\gamma_n\in \Gamma \andeqn \bigcap_{j=1}^n \textup{srs}_X(\gamma_j) \neq \emptyset \Bigg\}\] and write $\Gamma \curvearrowright (\mathcal{R}_X,\subseteq)$ to refer to the induced pointwise action, which is order-preserving on the corresponding poset under inclusion.
\end{df}

\begin{rem}\label{saturated automatic for Rubin}
    If $X$ is locally compact and if $\Gamma \curvearrowright X$ is of full support, then $X$ remains locally compact in the saturated topology and in this case $\mathcal{R}_X$ forms a basis for the saturated topology on $X$. If $\Gamma \curvearrowright X$ is Rubin, then all open subsets are saturated by \cite[Proposition 3.2]{BelEllMat_Rubin_2022} and $\textup{srs}_X(\gamma) = \textup{rsupp}_X(\gamma)$ for all $\gamma \in \Gamma$.
\end{rem}
We will consider convergent ultrafilters on $(\mathcal{R}_X,\subseteq)$ with respect to the saturated topology. Thus, for the saturated topology to be well-defined, we have to assume the full support condition. By \cite[Proposition 3.3. (3)]{BelEllMat_Rubin_2022}, an ultrafilter $\mathcal{F}\subseteq \mathcal{R}_X$ converges to $x$ in the saturated topology if and only if $\mathcal{F}$ contains all saturated neighborhoods of $x$.\par In the next proposition, we highlight a weaker form of the locally moving property that allows us to express convergence of ultrafilters on $\mathcal{R}_X$ in terms of localized subgroups.

\begin{prop}\label{order embedding}
    Let $\Gamma$ be a group and let $(\mathcal{S}(\Gamma),\subseteq)$ be the poset of all subgroups ordered by inclusion. Let $\Gamma \curvearrowright (\mathcal{S}(\Gamma),\subseteq)$ denote the order-preserving action by conjugating subgroups. Let $\Gamma \curvearrowright X$ be an action of full support such that the localized subgroup $\Gamma_W$ (see \autoref{localized subgroups}) is nontrivial for all nonempty saturated open subsets of the form $W=\textup{srs}_X(\gamma)\setminus \overline{\textup{srs}_X(\tau)}^{Sat}$ for some $\gamma,\tau\in \Gamma$. Then taking localized subgroups \[\Gamma_{(-)}\colon (\mathcal{R}_X,\subseteq)\to (\mathcal{S}(\Gamma),\subseteq)\] is a $\Gamma$-equivariant order embedding.
\end{prop}

\begin{proof}
     Given $U,V\in \mathcal{R}_X$, we first claim that $U \subseteq V$ if and only if $\Gamma_U \subseteq \Gamma_V$. For its verification, we follow the proof of \cite[Proposition 3.1]{BelEllMat_Rubin_2022}. Since taking localized subgroups commutes with taking intersections, it suffices to consider $U= \textup{srs}_X(\gamma)$ and $V=\textup{srs}_X(\tau)$ for some $\gamma,\tau\in \Gamma$. By definition $U \subseteq V$ implies $\Gamma_U \subseteq \Gamma_V$. Conversely, if $U \not\subseteq V$, then $U \not\subseteq \overline{V}^{Sat}$ by saturated regularity. Hence $W:=U\setminus \overline{V}^{Sat}$ is a nonempty open set of the required form and by assumption we can choose a $\gamma_0\in \Gamma_W\setminus\{1\}$ with $\textup{srs}_X(\gamma_0)\subseteq U\setminus V$, that is, $\gamma_0\in \Gamma_{U}\setminus \Gamma_{V}$ and this shows the claim.\par Finally note that for any $U\in \mathcal{R}_X$ and $\gamma,\tau\in \Gamma$, we have $\tau\in \Gamma_U$ if and only if $\textup{srs}_X(\gamma\tau\gamma^{-1}) =\gamma \cdot \textup{srs}_X(\tau)\subseteq \gamma \cdot U$ showing $\gamma\Gamma_U\gamma^{-1} = \Gamma_{\gamma \cdot U}$. 
\end{proof}
For a Rubin action $\Gamma\curvearrowright X$ and a faithful action $\Delta \curvearrowright Y$, we aim to formalize what it means for an embedding $\Phi\colon \Gamma\to \Delta$ to preserve the localized subgroup structure well enough to admit a unique anchor map in the sense of \autoref{anchor for phi df}. The necessary tools to relate ultrafilter limits in $Y$ and $X$ spatially are that $\Phi$ is locally regular and that $\Phi(\Gamma)\curvearrowright Y$ meets the requirements of \autoref{order embedding}.\par Note, however, that by Rubin's theorem the action $\Phi(\Gamma)\curvearrowright Y$ is Rubin if and only if the anchor map is a homeomorphism. To be able to construct anchor maps that are not necessarily homeomorphisms, we require $\Phi(\Gamma)\curvearrowright Y$ to be locally dense on saturated open sets only.

\begin{df}\label{Rubin emb df}
    Let $\Gamma \curvearrowright X$ be a Rubin action, let $Y$ be a locally compact Hausdorff space with no isolated points, and let $\Delta \subseteq\textup{Homeo}(Y)$ be a subgroup. Let $\Phi\colon \Gamma \rightarrow \Delta$ be an injective group homomorphism. Then we say that $\Phi$ is
    \begin{itemize}
        \item of \textit{full support} if $\Phi(\Gamma)\curvearrowright Y$ has no global fixed points, and
        \item \textit{locally dense on saturated subsets} if for all $y\in Y$ and every saturated open neighborhood $V$ of $y$, there exists $\tau\in \Gamma$ such that $y\in \textup{srs}_Y(\Phi(\tau))\subseteq V$.
    \end{itemize}
    We call $\Phi$ a \textit{Rubin embedding} if it is locally regular, of full support, and locally dense on saturated subsets.
\end{df}

\begin{rem}\label{Rub emb rem}
    It follows directly from \autoref{Rubin df} that the identity homomorphism $\Phi = \textup{id}_\Gamma$ is a Rubin embedding for any Rubin action $\Gamma \curvearrowright X$. The full support condition is of technical nature and is in this context equivalent to the existence of saturated neighborhoods. Note that the restricted action $\Phi(\Gamma) \curvearrowright Y$ does not have to be Rubin in general since not all open sets have to be saturated.
\end{rem}
The notion of a Rubin embedding is designed to make the following construction work.

\begin{prop}\label{P map}
    Let $\Gamma \curvearrowright X$ be a Rubin action, let $\Delta\subseteq\textup{Homeo}(Y)$ be a subgroup and let $\Phi\colon \Gamma \rightarrow \Delta$ be a Rubin embedding. We consider $\mathcal{R}_Y$ for the action $\Phi(\Gamma) \curvearrowright Y$ as in \autoref{saturated top}. Then, for any $V= \bigcap_{j} \textup{srs}_Y(\Phi(\gamma_j))\in \mathcal{R}_Y$, the set $P(V):= \bigcap_j \textup{srs}_X(\gamma_j)\in \mathcal{R}_X$ is well-defined and the resulting map $P\colon \mathcal{R}_Y \to \mathcal{R}_X$ is a $\Phi$-equivariant order isomorphism.
\end{prop}

\begin{proof}
    We claim that the actions $\Gamma \curvearrowright X$ and $\Phi(\Gamma)\curvearrowright Y$ both meet the requirements of \autoref{order embedding}. For $\Gamma \curvearrowright X$, this follows from the locally moving property. For $\Phi(\Gamma)\curvearrowright Y$, let $\gamma,\tau\in \Gamma$ be such that $W= \textup{srs}_Y(\Phi(\gamma))\setminus \overline{\textup{srs}_Y(\Phi(\tau))}^{Sat}$ is nonempty. Since $\textup{srs}_Y(\Phi(\gamma))$ is saturated open and $\overline{\textup{srs}_Y(\Phi(\tau))}^{Sat}$ is saturated closed, $W$ is a saturated open set as well. Consequently, local density on saturated subsets applies to any $y\in W$ and we have that $\Phi(\Gamma)_W\neq \{1\}$ since $y$ is not isolated. This shows the claim that both $\Gamma_{(-)}$ and $\Phi(\Gamma)_{(-)}$ are equivariant order embeddings by \autoref{order embedding}.\par From here, by local regularity of $\Phi$, taking image subgroups $\Phi_*\colon \mathcal{S}(\Gamma)\to \mathcal{S}(\Phi(\Gamma))$ restricts to a $\Phi$-equivariant order isomorphism between the respective posets of localized subgroups. Thus $P = \Gamma_{(-)}^{-1}\circ \Phi_*^{-1}\circ \Phi(\Gamma)_{(-)}$ is a well-defined $\Phi$-equivariant order isomorphism, as claimed.
\end{proof}
For any $y\in Y$, there exist ultrafilters on $\mathcal{R}_Y$ that converge to $y$ in the saturated topology, since local compactness of $Y$ and full support of $\Phi$ imply that $y$ admits a saturated neighborhood with compact closure and thus \cite[Proposition 3.4. (2)]{BelEllMat_Rubin_2022} applies. In the following proposition, we use the $P$-map in \autoref{P map} for the desired anchor map construction and define the corresponding point $\rho(y)\in X$.

\begin{prop}\label{rho y construction}
    Let $\Gamma \curvearrowright X$ be a Rubin action, let $\Delta\subseteq\textup{Homeo}(Y)$ be a subgroup and let $\Phi\colon \Gamma \rightarrow \Delta$ be a Rubin embedding. Let $y\in Y$ and let $\mathcal{F}\subseteq\mathcal{R}_Y$ be an ultrafilter that converges to $y$ in the saturated topology. Then the ultrafilter $P(\mathcal{F})\subseteq \mathcal{R}_X$ converges to a unique point $\lim(P(\mathcal{F}))\in X$. Moreover, for any $U\in \mathcal{R}_Y$, we have $\lim(P(\mathcal{F}))\in P(U)$ if and only if $y\in U$.\par In particular, the limit point does only depend on $y$ and not on the choice of $\mathcal{F}$, and the assignment $\rho(y):=\lim(P(\mathcal{F}))\in X$ is well-defined. 
\end{prop}

\begin{proof}
    Since $P$ is an order isomorphism by \autoref{P map}, $P(\mathcal{F})$ is an ultrafilter on $\mathcal{R}_X$. Since $\Phi$ is of full support and $X$ and $Y$ are locally compact, there exists $\gamma\in \Gamma$ such that $y\in \textup{srs}_Y(\Phi(\gamma))$ and $\supp_X(\gamma)\subseteq X$ is compact. As a result, we have that $\textup{srs}_X(\gamma) = P(\textup{srs}_Y(\Phi(\gamma))) \in P(\mathcal{F})$ and thus the image ultrafilter converges in $X$. Since $\Gamma\curvearrowright X$ is Rubin, $X$ is Hausdorff with respect to the saturated topology by \autoref{saturated automatic for Rubin} and the set of limit points $\bigcap_{U\in \mathcal{F}} \ov{P(U)}$ is a singleton, say $\lim(P(\mathcal{F}))\in X$.\par For the second part, let $U\in \mathcal{R}_Y$. Since $\Gamma\curvearrowright X$ is Rubin by assumption, \cite[Proposition 3.5 (2)]{BelEllMat_Rubin_2022} applies and we have that $\lim(P(\mathcal{F}))\in P(U)$ if and only if there is $V\in \mathcal{R}_Y$ such that $P(V)\subseteq P(U)$ and $(\mathcal{R}_X)_{\subseteq P(V)} \subseteq \Gamma_{P(U)}\cdot P(\mathcal{F})$. We aim to express the condition $y\in U$ similarly in terms of orbits under localized subgroups. Concretely, we claim that $y\in U$ if and only if there is $V\in (\mathcal{R}_Y)_{\subseteq U} \textup{\,\,such that \,\,} (\mathcal{R}_Y)_{\subseteq V}\subseteq \Phi(\Gamma)_U\cdot \mathcal{F}$. If $y\in U$, then, since $\Phi$ is locally dense on saturated subsets, there is $\Phi(\tau)\in \Phi(\Gamma)_U$ such that we may choose $V:=\textup{srs}_Y(\Phi(\tau))\in (\mathcal{R}_Y)_{\subseteq U}$. For any $W\in (\mathcal{R}_Y)_{\subseteq V}$, there is $\Phi(\gamma)\in \Phi(\Gamma)_U$ such that $\Phi(\gamma)\cdot y\in W$. Since $\mathcal{F}$ converges to $y$, it contains all saturated neighborhoods of $y$. This implies that $\Phi(\gamma)^{-1}\cdot W\in \mathcal{F}$ and therefore $W\in \Phi(\Gamma)_U \cdot \mathcal{F}$.\par Conversely, assuming the existence of $V$ with the orbit condition, we can choose $W\in (\mathcal{R}_Y)_{\subseteq V}$ such that $\ov{P(W)}$ is a compact subset of $P(V)$ in $X$. By \autoref{P map}, this implies that $\ov{W}^{Sat}\subseteq V\subseteq U$. By assumption, we have $\Phi(\gamma)\cdot W\in \mathcal{F}$ for some $\Phi(\gamma)\in \Phi(\Gamma)_U$ and thus $y\in \ov{\Phi(\gamma)\cdot W}^{Sat} = \Phi(\gamma)\cdot \ov{W}^{Sat}\subseteq \Phi(\gamma)\cdot U\subseteq U$, as desired. This shows the claimed equivalence. Applying the $\Phi$-equivariant order isomorphism $P\colon \mathcal{R}_Y\to \mathcal{R}_X$ in \autoref{P map}, we observe $P(\Phi(\Gamma)_U\cdot \mathcal{F}) = \Gamma_{P(U)}\cdot P(\mathcal{F})$ and hence we obtain the following chain of equivalences:
    \begin{align*}
        y\in U 
        &\iff \textup{there is\,\,} V\in (\mathcal{R}_Y)_{\subseteq U} \textup{\,\,such that\,} (\mathcal{R}_Y)_{\subseteq V}\subseteq \Phi(\Gamma)_U.\mathcal{F}^y\\
        &\iff \textup{there is\,\,} V\in (\mathcal{R}_Y)_{\subseteq U} \textup{\,\,such that } (\mathcal{R}_X)_{\subseteq P(V)}\subseteq \Gamma_{P(U)}.P(\mathcal{F}^y)\\
        &\iff \lim(P(\mathcal{F}))\in P(U).
    \end{align*}
    In particular, if $\mathcal{F}'\subseteq\mathcal{R}_Y$ is any other ultrafilter that converges to $y$ in the saturated topology, then $\lim(P(\mathcal{F}))$ and $\lim(P(\mathcal{F}'))$ share the same neighborhoods in $P(\mathcal{R}_Y)=\mathcal{R}_X$ and therefore agree in $X$ by \autoref{saturated automatic for Rubin}. This shows that $\rho(y):=\lim(P(\mathcal{F}))\in X$ is well-defined. 
\end{proof}

\begin{prop}\label{rho map}
    Let $\Gamma \curvearrowright X$ be a Rubin action, let $\Delta\subseteq\textup{Homeo}(Y)$ be a subgroup and let $\Phi\colon \Gamma \rightarrow \Delta$ be a Rubin embedding. Then the assignment $y\mapsto \rho(y)$ described in \autoref{rho y construction} defines a surjective anchor map $\rho\colon Y\to X$ for $\Phi$.
\end{prop}

\begin{proof}
    Let $\gamma\in \Gamma$ and $y\in Y$. Let $\mathcal{F}\subseteq\mathcal{R}_Y$ be an ultrafilter that converges to $y$ and let $\mathcal{G}\subseteq\mathcal{R}_Y$ be an ultrafilter that converges to $\Phi(\gamma)\cdot y$ in the saturated topology. The $\Phi$-equivariance of $\rho$ follows from the equivariance of the order isomorphism in \autoref{P map} since $\gamma\cdot P(\mathcal{F})$ has the same limit as $P(\mathcal{G})$.\par By \autoref{rho y construction}, for all $U\in \mathcal{R}_Y$, we have $y\in U$ if and only if $\rho(y)\in P(U)$. For $\gamma \in \Gamma$ and $U= \textup{srs}_Y(\Phi(\gamma))$, we obtain \[\rho^{-1}(\textup{srs}_X(\gamma)) = P^{-1}(\textup{srs}_X(\gamma)) = \textup{srs}_Y(\Phi(\gamma)).\] This shows both the desired anchor property in \autoref{anchor for phi df} and continuity of $\rho$ since any open set in $X$ is saturated by assumption.\par To see surjectivity, let $x\in X$, let $\mathcal{F}\subseteq \mathcal{R}_X$ be an ultrafilter converging to $x$, and let $Q_x:= \bigcap_{U\in \mathcal{F}} \overline{P^{-1}(U)}^{Sat}$ be the set of limits of $P^{-1}(\mathcal{F})$. We claim that the set $Q_x$ is nonempty. Indeed, since $\mathcal{F}$ is convergent, it contains a set $C\in \mathcal{F}^x$ with compact closure. For finitely many $U_i\in (\mathcal{F})_{\subseteq C}$, the intersection $\bigcap_i U_i$ is nonempty by the filter property, and thus $\bigcap_i P^{-1}(U_i)$ is nonempty by \autoref{P map}. Hence also $Q_x$ is nonempty by compactness of $\ov{P^{-1}(C)}^{Sat}$. Finally, we claim that $Q_x\subseteq \rho^{-1}(\{x\})$. Indeed, for any $V\in \mathcal{R}_Y$ and $y\in \ov{V}^{Sat}$, we have $\rho(y)\in \ov{P(V)}$ by \autoref{rho y construction} and continuity of $\rho$. Combined with the assumption that $\Gamma\curvearrowright X$ is Rubin, this shows that \[\rho(Q_x) \subseteq \bigcap_{U\in \mathcal{F}} \overline{U} = \{x\}.\qedhere\] 
\end{proof}
The following is our main result.

\begin{thm}\label{rubin emb char}
    Let $X$ and $Y$ be locally compact Hausdorff spaces with no isolated points, let $\Gamma \curvearrowright X$ be a Rubin action, let $\Delta\curvearrowright Y$ be a faithful action, and let $\Phi\colon \Gamma \rightarrow \Delta$ be an injective group homomorphism. The following are equivalent:
    \begin{enumerate}
        \item $\Phi$ is a Rubin embedding.
        \item There is a unique surjective anchor map $\rho\colon Y\to X$ for $\Phi$.
    \end{enumerate}
\end{thm}

\begin{proof}
    \autoref{rho map} and the uniqueness part of \autoref{spatial map implies lr} show that $(1)$ implies $(2)$. Conversely, let $\rho\colon Y\to X$ be a surjective anchor map for $\Phi$. By the first part of \autoref{spatial map implies lr}, the embedding $\Phi$ is locally regular. Note $\Phi$-equivariance implies that $\rho$ maps the set of global fixed points for $\Phi(\Gamma)\curvearrowright Y$ to the set of global fixed points for $\Gamma\curvearrowright X$. Since the latter action is Rubin, both sets are empty and therefore $\Phi$ is of full support. It remains to show that $\Phi$ is locally dense on saturated subsets. Let $y\in Y$ and let $U\subseteq Y$ be a saturated open neighborhood of $y$. By local compactness there is $U_0\in \mathcal{R}_Y$ such that $y\in U_0\subseteq U$ and since $\rho$ maps $U_0$ to an open neighborhood of $\rho(y)$, we can apply local density of $\Gamma\curvearrowright X$ to find $\tau\in \Gamma$ such that $\rho(y)\in \textup{srs}_X(\tau)\subseteq \rho(U_0)$. Finally, the anchor property $\rho^{-1}(\textup{srs}_X(\tau))=\textup{srs}_Y(\Phi(\tau))$ shows that $y\in \textup{srs}_Y(\Phi(\tau))\subseteq U_0\subseteq U$. This shows local density on saturated subsets.
\end{proof}

\begin{eg}\label{Rubin emb between BT groups}
    Applying tables coordinatewise, duplicating them or permuting them provide canonical ways to embed generalized Brin-Thompson groups into others with potentially more alphabet factors. We illustrate this idea with the example of the embedding $\iota\colon V_2\hookrightarrow 2V_2$ given by $\iota \left(\psi\right)(v_1^{(j)}x,y) := (u_1^{(j)}x,y)$ for any table $\begin{pmatrix} v_1\\ u_1\end{pmatrix}$ of $X_2$ implementing $\psi$ and $x,y\in X_2$. That is, using empty words for all $v_2^{(j)}$ and $u_2^{(j)}$ constitutes a table of $X_2\times X_2$ implementing $\iota(\psi)$. We claim that $\iota$ is a Rubin embedding with the projection onto the first coordinate $\pi_1\colon X_2\times X_2 \to X_2$ as its anchor map. To see this, it is immediate that $\pi_1\colon X_2\times X_2 \to X_2$ is continuous, surjective and $\iota$-equivariant. Furthermore, the regular support of any homeomorphism $\overline{\begin{pmatrix} v\\ u\end{pmatrix}}\in V_{k_1,\dots,k_m}$ is $\bigsqcup_{j\colon u^{(j)}\neq v^{(j)}} u^{(j)}\prod_{i=1}^m X_{k_i}$, since the periodic fixed points under nonidentical prefixes are nowhere dense. Hence, we get $\textup{srs}_{X_2}(\psi) = \bigsqcup_{j\colon u_1^{(j)}\neq v_1^{(j)}} u_1^{(j)}X_2$ and $\textup{rsupp}_{X_2\times X_2}(\iota(\psi)) =  \bigsqcup_{j\colon u^{(j)}\neq v^{(j)}} u_1^{(j)}X_2\times X_2$, which agrees with $\textup{srs}_{X_2\times X_2}(\iota(\psi))$ since it is a saturated clopen set. This shows that $\pi_1^{-1}(\textup{srs}_{X_2}(\psi)) = \textup{srs}_{X_2\times X_2}(\iota(\psi))$ and thus $\iota$ is locally regular by \autoref{spatial map implies lr}.
\end{eg}

\begin{rem}\label{tfg actor remark}
    Motivated by \cite[Theorem 6.11]{GarGun_lpemb_2026}, the equivalence in \autoref{rubin emb char} allows us to describe groupoid actors between minimal effective \'etale groupoids with Cantor unit spaces solely in terms of sufficiently well-behaved embeddings between the associated topological full groups that encode the corresponding anchor maps implicitly. This relies on the groupoid model by germs of the action by the topological full group as well as the characterization of actors between transformation groupoids. See also \cite[Example 3.6]{Tay_Fellactor_2023}.
\end{rem}

\end{document}